\documentclass[a4paper, 11pt]{article}

\usepackage[english]{babel}
\usepackage{amsmath}
\usepackage{amsthm} 
\usepackage{graphicx}
\usepackage{pinlabel}
\usepackage[alwaysadjust]{paralist}
\usepackage[%
breaklinks,
colorlinks=true,
linkcolor=black,
anchorcolor=black,
citecolor=black,
filecolor=black,
menucolor=black,
urlcolor=black]{hyperref}
\usepackage[margin=\parindent]{caption}
\usepackage[OT2,T1]{fontenc}% for Milnor's Lobachevsy function $\ML$
\usepackage[charter]{mathdesign}

\newcommand{\R}{\mathbb{R}}
\newcommand{\HH}{\mathbb{H}}
\newcommand{\arsinh}{\operatorname{arsinh}}
\usepackage[T2A,T1]{fontenc}
\newcommand{\Cl}{\text{Cl}_2}
\newcommand{\Clh}{\text{Clh}_2}
\newcommand{\cL}{\mbox{\fontencoding{OT2}\fontfamily{wncyr}\fontseries{m}\fontshape{n}\selectfont L}}
\newcommand{\ellbar}{\bar{\ell}}

\theoremstyle{definition}
\newtheorem{definition}{Definition}[section]

\newtheorem{remark}[definition]{Remark}
\theoremstyle{plain}
\newtheorem{theorem}{Theorem}
\newtheorem*{theorem*}{Theorem}
\newtheorem{oldtheorem}{Theorem}

\newtheorem{proposition}[definition]{Proposition}
\newtheorem{corollary}[definition]{Corollary}
\newtheorem{lemma}[definition]{Lemma}

\title{A variational principle for cyclic polygons with prescribed
  edge lengths}
\author{Hana Kou\v{r}imsk\'{a}, Lara Skuppin, Boris Springborn}
\date{}

\begin{document}

\maketitle

\begin{abstract}
  We provide a new proof of the elementary geometric theorem on the
  existence and uniqueness of cyclic polygons with prescribed side
  lengths. The proof is based on a variational principle involving the
  central angles of the polygon as variables. The uniqueness follows
  from the concavity of the target function. The existence
  proof relies on a fundamental inequality of information
  theory.

  We also provide proofs for the corresponding theorems of spherical
  and hyperbolic geometry (and, as a byproduct, in $1+1$
  spacetime). The spherical theorem is reduced to the Euclidean
  one. The proof of the hyperbolic theorem treats three cases
  separately: Only the case of polygons inscribed in compact circles
  can be reduced to the Euclidean theorem. For the other two cases,
  polygons inscribed in horocycles and hypercycles, we provide
  separate arguments. The hypercycle case also proves the theorem for
  ``cyclic'' polygons in $1+1$ spacetime. 

  \bigskip\noindent%
  \href{http://www.ams.org/mathscinet/search/mscdoc.html?code=52B60}{52B60}
\end{abstract}

\section{Introduction}

This article is concerned with cyclic polygons, i.e., convex
polygons inscribed in a circle. We
will provide a new proof of the following elementary theorem in Section~\ref{sec:proof_euc}. 

\label{sec:intro}
\begin{theorem}
  \label{thm:euc}
  There exists a Euclidean cyclic polygon with $n\geq 3$ sides of lengths
  $\ell_{1},\ldots,\ell_n\in\R_{>0}$ if and only if they satisfy the
  polygon inequalities
  \begin{equation}
    \label{eq:polygon}
    \ell_{k}<\sum_{\substack{i=1\\i\not=k}}^{n}\ell_{i}\,,
  \end{equation}
  and this cyclic polygon is unique.
\end{theorem}

Our proof involves a variational principle with the central angles as
variables. The variational principle has a geometric interpretation in
terms of volume in $3$-dimensional hyperbolic space (see
Remark~\ref{rem:hyperbolic_volume}).
Another striking feature of our proof is the use of a
fundamental inequality of information theory:

\begin{theorem*}[``Information Inequality'']
  Let $p=(p_1, \ldots, p_m)$ and $q=(q_1, \ldots, q_m)$ be discrete
  probability distributions, then
\begin{equation}
\label{eq:information}
\sum_{k=1}^{m} p_k\log\frac{p_k}{q_k} \geq 0,
\end{equation} 
and equality holds if and only if $p=q$.
\end{theorem*}

The left hand side of inequality~\eqref{eq:information} is called the
\emph{Kullback--Leibler divergence} or \emph{information gain} of
$q$ from $p$, also the \emph{relative entropy} of $p$ with respect to
$q$. The inequality follows from the strict concavity of the logarithm
function (see, e.g., Cover and Thomas~\cite{Cover1991}).

In Sections~\ref{sec:proof_sph} and~\ref{sec:proof_hyp} we provide
proofs for non-Euclidean versions of Theorem~\ref{thm:euc}. The
spherical version requires an extra inequality:

\begin{theorem}
  \label{thm:sph}
  There exists a spherical cyclic polygon with $n\geq 3$ sides of
  lengths $\ell_{1},\ldots,\ell_n\in\R_{>0}$ if and only if they
  satisfy the polygon inequalities~\eqref{eq:polygon} and
  \begin{equation}
    \label{eq:polygon_sph}
    \sum_{i=1}^{n}\ell_{i} < 2\pi\,,
  \end{equation}
  and this cyclic spherical polygon is unique.
\end{theorem}

Inequality~\eqref{eq:polygon_sph} is necessary because the
perimeter of a circle in the unit sphere cannot be greater than
$2\pi$, and the perimeter of the inscribed polygon is a lower
bound. We require strict inequality to exclude polygons that
degenerate to great circles (with all interior angles equal to $\pi$).

In Section~\ref{sec:proof_sph}, we prove Theorem~\ref{thm:sph} by a
straightforward reduction to Theorem~\ref{thm:euc}: Connecting the
vertices of a spherical cyclic polygon by straight line segments in
the ambient Euclidean $\R^{3}$, one obtains a Euclidean cyclic
polygon. 

In the case of hyperbolic geometry, the notion of ``cyclic polygon''
requires additional explanation. We call a convex hyperbolic polygon
\emph{cyclic} if its vertices lie on a curve of constant non-zero
curvature. Such a curve is either
\begin{compactitem}
\item a hyperbolic circle if the curvature is greater than $1$,
\item a horocycle if the curvature is equal to $1$,
\item a hypercycle, i.e., a curve at constant distance from a geodesic
  if the curvature is strictly between $0$ and $1$.
\end{compactitem}
\begin{theorem}
  \label{thm:hyp}
  There exists a hyperbolic cyclic polygon with $n\geq 3$ sides of lengths 
  $\ell_{1},\ldots,\ell_n\in\R$ if and only if they satisfy the
  polygon inequalities~\eqref{eq:polygon}, and this cyclic hyperbolic
  polygon is unique.
\end{theorem}

We prove this theorem in Section~\ref{sec:proof_hyp}. The case of
hyperbolic polygons inscribed in circles can be reduced to
Theorem~\ref{thm:euc} by considering the hyperboloid model of the
hyperbolic plane: Connecting the vertices of a hyperbolic polygon
inscribed in a circle by straight line segments in the ambient
$\R^{2,1}$, one obtains a Euclidean cyclic polygon.

The cases of polygons inscribed in horocycles and hypercycles cannot
be reduced to the Euclidean case because the intrinsic geometry of the
affine plane of the polygon is not Euclidean: In the horocycle case,
the scalar product is degenerate with a $1$-dimensional kernel. Hence,
this case reduces to the case of degenerate polygons inscribed in a
straight line. It is easy to deal with. In the hypercycle case, the
scalar product is indefinite. This case reduces to polygons inscribed
in hyperbolas in flat $1+1$ spacetime. The variational principle of
Section~\ref{sec:proof_euc} can be adapted for this case (see
Section~\ref{sec:concluding}), but the corresponding target function
fails to be concave or convex. It may be possible to base a proof of
existence and uniqueness on this variational principle, perhaps using
a $\min$-$\max$-argument, but we do not pursue this route in this
article. Instead, we deal with polygons inscribed in hypercycles using
a straightforward analytic argument.

\paragraph{Some history, from ancient to recent.}

Theorems~\ref{thm:euc}--\ref{thm:hyp} belong to the circle of results
connected with the classical isoperimetric problem. 
As the subject is ancient and the body of literature is vast, we can
only attempt to provide a rough historical perspective and ask for
leniency regarding any essential work that we fail to
mention.

The early history of the relevant results about polygons is briefly
discussed by Steinitz~\cite{Steinitz} (section 16). Steinitz goes on
to discuss analogous results for polyhedra, a topic into which we will not
go. A more recent and comprehensive survey of proofs of
the isoperimetric property of the circle was given by
Bl{\aa}sj{\"o}~\cite{Blasjo2005}. 

It was known to Pappus that the regular $n$-gon had the largest area
among $n$-gons with the same perimeter, and that the area grew with
the number of sides. This was used to argue for the isoperimetric
property of the circle:

\begin{oldtheorem}[Isoperimetric Theorem]
  \label{thm:isoperimetric}
  Among all closed planar curves with given length, only the circle
  encloses the largest area.
\end{oldtheorem}

It is not clear who first stated the following theorem about polygons:

\begin{oldtheorem}[Secant Polygon]
  \label{thm:secant}
  Among all $n$-gons with given side lengths, only the one inscribed in a
  circle has the largest area.
\end{oldtheorem}

This was proved by Moula~\cite{Moula1737}, by
L'Huilier~\cite{LHuilier1782} (who cites Moula), and by
Steiner~\cite{Steiner1842} (who cites L'Huilier). L'Huilier also
proved the following theorem:

\begin{oldtheorem}[Tangent Polygon]
  \label{thm:tangent}
  Among all convex $n$-gons with given angles, only the one
  circumscribed to a circle has the largest area when the perimeter is
  fixed and and smallest perimeter when the area is fixed.
\end{oldtheorem}

Steiner also proves versions of Theorems~\ref{thm:secant}
and~\ref{thm:tangent} for spherical polygons. None of these authors
deemed it necessary to prove the existence of a maximizer, an issue
that became generally recognized only after
Weierstrass~\cite{Weierstrass_Werke_7}. For polygons, the existence of
a maximizer follows by a standard compactness argument.

Blaschke~\cite{Blaschke1916} ({\S} 12) notes that the quadrilateral
case ($n=4$) of Theorem~\ref{thm:secant} can easily be deduced from
the Isoperimetric Theorem~\ref{thm:isoperimetric} using Steiner's
four-hinge method. Conversely, one can similarly deduce
Theorem~\ref{thm:isoperimetric} and the general
Theorem~\ref{thm:secant} from the quadrilateral case of
Theorem~\ref{thm:secant}. He remarks that the quadrilateral case of
Theorem~\ref{thm:secant} can be proved directly by deriving the
following equation for the area $A$ of a quadrilateral with sides
$\ell_{k}$:
\begin{equation}
  \label{eq:area_quad}
  A^{2} = (s-\ell_{1})(s-\ell_{2})(s-\ell_{3})(s-\ell_{4})
  - \ell_{1}\ell_{2}\ell_{3}\ell_{4}\cos^{2}\theta,
\end{equation}
where $s=(\ell_{1}+\ell_{2}+\ell_{3}+\ell_{4})/2$ is half the
perimeter, and $\theta$ is the arithmetic mean of two opposite angles.

Neither Blaschke, nor Steiner, L'Huilier, or Moula provide an argument
for the uniqueness of the maximizer in Theorems~\ref{thm:secant}
or~\ref{thm:tangent}. It seems that even after Weierstrass, the fact
that the sides determine a cyclic polygon uniquely was considered too
obvious to deserve a proof.

Penner~\cite{Penner1987} (Theorem 6.2) gives a complete proof of
Theorem~\ref{thm:euc}. He proceeds by showing that there is one and
only one circumcircle radius that allows the construction of a
Euclidean cyclic polygon with given sides (provided they satisfy the
polygon inequalities). 

Schlenker~\cite{Schlenker2007} proves Theorems~\ref{thm:sph}
and~\ref{thm:hyp}, and also the isoperimetric property of
non-Euclidean cyclic polygons, i.e., the spherical and hyperbolic
versions of Theorem~\ref{thm:secant}. His proofs of the isoperimetric
property are based on the remarkable equation
\begin{equation}
  \label{eq:schlenker}
  \sum \dot\alpha_{i}v_{i}=0
\end{equation}
characterizing the change of angles $\alpha_{i}$ of a spherical or
hyperbolic polygon under infinitesimal deformations with fixed side
lengths. Here, $v_{i}\in\R^{3}$ are the position vectors of the
polygon's vertices in the sphere or in the hyperboloid,
respectively. To prove the uniqueness of spherical and hyperbolic
cyclic polygons with given sides he uses separate arguments similar
to Penner's.

\section{Euclidean polygons. Proof of Theorem~\ref{thm:euc}}
\label{sec:proof_euc}

To
construct an inscribed polygon with given side lengths $\ell=(\ell_{1},\ldots,\ell_n)\in\R_{>0}^n$ (see Figure~\ref{fig:eucl_polygon}) is
equivalent to finding a point $(\alpha_1,\ldots\alpha_{n})$
in the set 
\begin{equation}
D_n=\big\{\alpha \in \R_{>0}^n \;\big|\;\sum_{k=1}^n \alpha_k = 2\pi \big\}
\subset\R^{n}
\end{equation}
satisfying, for some $R\in\R$ and for all $k\in\{1,\ldots,n\}$,
\begin{equation}\label{eq:eucl_radius}
  \frac{\ell_k}{2}=R\sin\frac{\alpha_k}{2}.
\end{equation}
%for some $R\in\R$ (see Figure~\ref{fig:eucl_polygon}).

\begin{figure}%[h]
\labellist
\small\hair 2pt
 \pinlabel {$\ell_1$} [ ] at 23 89
 \pinlabel {$\ell_2$} [ ] at 31 21
 \pinlabel {$\dots$} [ ] at 77 55
 \pinlabel {$\ell_n$} [ ] at 92 100
 \pinlabel {$\alpha_1$} [ ] at 56 68
 \pinlabel {$\alpha_2$} [ ] at 56 53
 \pinlabel {$\alpha_n$} [ ] at 73 72
 \pinlabel {$R$} [ ] at 92 63
\endlabellist
\centering
\includegraphics[scale=1.0]{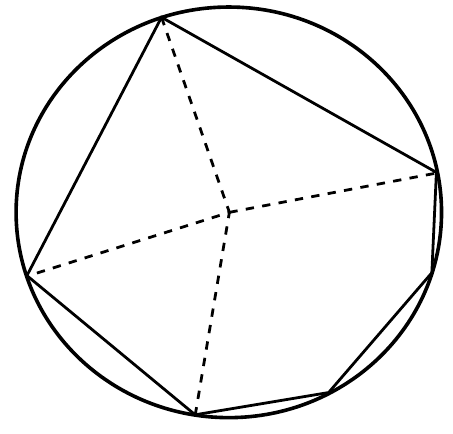}
\caption{Euclidean polygon inscribed in a circle}
\label{fig:eucl_polygon}
\end{figure}
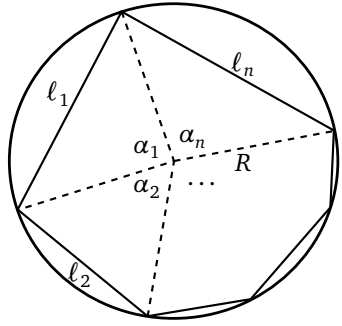

This problem admits the following variational formulation. Define the
function $f_\ell:\R^n \to \R$ by
\begin{equation}
  \label{eq:f_ell}
  f_\ell(\alpha) = 
  \sum_{k=1}^n\big( \Cl(\alpha_k) +  \log(\ell_k)\,\alpha_k\big)
\end{equation}
where $\Cl$ denotes Clausen's integral~\cite{Lewin1981}:
\begin{equation}
  \Cl (x)= -\int_0^x \log\Big|2\sin \frac{t}{2}\Big|\,dt.
\end{equation}
Clausen's integral is closely related to Milnor's Lobachevsky function \cite{Milnor1982}:
\begin{equation*}
  \cL(x)=\frac{1}{2}\:\Cl(2x).
\end{equation*}
The function $\Cl:\R\rightarrow\R$ is continuous, $2\pi$-periodic, and
odd. It is differentiable except at integer multiples of $2\pi$ where
the graph has vertical tangents (see Figure~\ref{fig:cl2}).
\begin{figure}[h]
  \centering
  \includegraphics{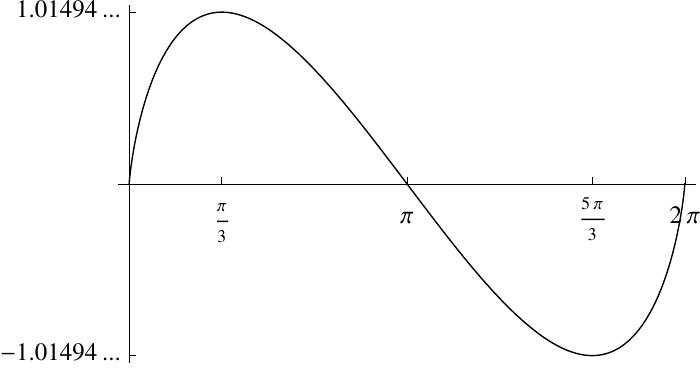}
  \caption{Graph of Clausen's integral $\Cl(x)$}
  \label{fig:cl2}
\end{figure}

\begin{proposition}[Variational Principle]\label{prop:radius_crit_pt}
  A point $\alpha \in D_n$ is a critical point of~$f_\ell$ restricted
  to~$D_{n}$ if and only if there exists an $R\in \R$ satisfying
  equations~\eqref{eq:eucl_radius}.
\end{proposition}

\begin{proof}
  A point $\alpha \in D_n$ is a critical point of $f_\ell$ restricted
  to $D_n$ if and only if there exists a Lagrange multiplier $\log R$
  such that $\nabla f_\ell(\alpha)=(\log R)\nabla g(\alpha)$ for the
  constraint function $g(\alpha)=\sum\alpha_k$, i.e.,
 \begin{equation*}
  \begin{pmatrix}
   -\log \left|2\sin \frac{\alpha_1}{2}\right| + \log \ell_1\\
   \vdots \\
   -\log \left|2\sin \frac{\alpha_n}{2}\right| + \log \ell_n
  \end{pmatrix}
  =
  \log R
  \begin{bmatrix}
   1\\
   \vdots \\
   1 
  \end{bmatrix}.
 \end{equation*}
Since $0<\alpha_k<2\pi$ we may omit the absolute value signs,
obtaining equations~\eqref{eq:eucl_radius}.
\end{proof}

Thus, to prove Theorem~\ref{thm:euc}, we need to show that $f_\ell$
has a critical point in $D_n$ if and only if the polygon
inequalities~\eqref{eq:polygon} are satisfied, and that this critical
point is then unique. The following proposition and corollary deal
with the uniqueness claim.

\begin{proposition}\label{prop:concavity}
 The function $f_\ell$ is strictly concave on~$D_n$.
\end{proposition}

\begin{corollary}
  If $f_\ell$ has a critical point in $D_n$, it is the unique
  maximizer of $f_{\ell}$ in the closure 
  $
  \bar D_n=\{\alpha \in \R_{\geq 0}^n \;|\;\sum\alpha_k = 2\pi \}.
  $
\end{corollary}

\noindent%
This proves the uniqueness claim of Theorem~\ref{thm:euc}.

\begin{proof}[Proof of Proposition~\ref{prop:concavity}]
  We will show that 
  \begin{equation}
    \label{eq:V_n}
    V_n(\alpha)=\sum_{k=1}^n \Cl(\alpha_k)
  \end{equation}
  is strictly concave on $D_{n}$. Since $V_{n}$ differs from
  $f_\ell$ by a linear function, this is equivalent to
  the claim.
  
  Rivin~\cite{Rivin1994} (Theorem 2.1) showed that $V_{3}$ is strictly
  concave on $D_{3}$. For $n>3$ we 
  proceed by induction on $n$ 
  by ``cutting off a triangle'': First, note the obvious identity
  \begin{multline*}
  V_n(\alpha_1,\dots, \alpha_n)= V_{n-1}(\alpha_1,\dots, \alpha_{n-1}+\alpha_n)
  -\Cl(\alpha_{n-1}+\alpha_n)\\+\Cl(\alpha_{n-1})+\Cl(\alpha_n).
 \end{multline*}
 Since Clausen's integral is $2\pi$-periodic and odd,
 \begin{align*}
   -\Cl(\alpha_{n-1}+\alpha_n)=\Cl(2\pi-\alpha_{n-1}-\alpha_n)=\Cl\left(\sum_{k=1}^{n-2}\alpha_k\right),
 \end{align*}
 so
 \begin{equation*}
   V_n(\alpha_1,\dots, \alpha_n)=V_{n-1}(\alpha_1,\dots, \alpha_{n-1}+\alpha_n)+V_3\left(\sum_{k=1}^{n-2}\alpha_k,\alpha_{n-1},\alpha_n \right).
 \end{equation*}
 Hence, if $V_{n-1}$ and $V_3$ are strictly concave on
 $D_{n-1}$ and $D_3$, respectively, the claim for $V_n$ follows.
\end{proof}

Since $f_{\ell}$ attains its maximum on the
compact set $\bar D_{n}$, it remains to show that the maximum is
attained in $D_{n}$ if and only if the polygon
inequalities~\eqref{eq:polygon} are satisfied. This is achieved by the
following Propositions~\ref{prop:max_boundary}
and~\ref{prop:max_vertex}.

Note that $\bar D_{n}$ is an $(n-1)$-dimensional simplex in $\R^{n}$.
Its vertices are the points $2\pi e_{1},\ldots,2\pi e_{n}$, where
$e_{k}$ are the canonical basis vectors of $\R^{n}$. The relative
boundary of the simplex $\bar D_n$ is 
\begin{equation}
  \label{eq:boundary}
  \partial \bar D_{n}=\{\alpha\in\bar D_{n}\;|\;\alpha_{k}=0 \text{
    for at least one } k\}.
\end{equation}

\begin{proposition}
  \label{prop:max_boundary}
  If the function $f_\ell$ attains its maximum on the
  simplex~$\bar D_{n}$ at a boundary point $\alpha\in\partial\bar
  D_{n}$, then $\alpha$ is a vertex.
\end{proposition}

\begin{proof}
  Suppose $\alpha\in\partial\bar D_{n}$ is not a vertex. We need to
  show that $f_{\ell}$ does not attain its maximum at $\alpha$. This
  follows from the fact that the derivative of $f_{\ell}$ in a direction
  pointing towards $D_{n}$ is $+\infty$. 

  Indeed, suppose $v\in\R_{\geq 0}^{n}$, $\sum_{k}v_{k}=0$ and
  $v_{k}>0$ if $\alpha_{k}=0$. Then $\alpha+tv\in D_{n}$ for
  small enough~$t>0$, and because $\lim_{x\rightarrow 0}\Cl'(x)=+\infty$,
  \begin{equation}
    \label{eq:lim}
    \lim_{t\rightarrow 0}\;\frac{d}{dt}\,f_{\ell}(\alpha+tv)=+\infty.
  \end{equation}
  Hence $f_{\ell}(\alpha+tv)>f_{\ell}(\alpha)$ for small enough $t>0$.
\end{proof}

\begin{proposition}
  \label{prop:max_vertex}
  The function $f_\ell$ attains its maximum on $\bar D_{n}$ at a vertex 
  $2\pi e_{k}$
  if and only if 
  \begin{equation}
    \label{eq:polygon_violated}
    \ell_{k}\geq\sum_{\substack{i=1\\i\not=k}}^{n}\ell_{i}\,.
  \end{equation}
\end{proposition}

\begin{proof}
  By symmetry, it is enough to consider the case $k=n$, i.e., to show
  that the function $f_\ell$ attains its maximum on $\bar D_{n}$ at the
  vertex $(0,\ldots,0,2\pi)$ if and only if
  $\ell_{n}\geq\sum_{k=1}^{n-1}\ell_{k}$. To this end, we will
  calculate the directional derivative of $f_{\ell}$ in directions
  $v\in\R^{n}$ pointing inside $D_{n}$, i.e., satisfying
  \begin{equation*}
    v_k \geq 0\quad\text{for}\quad k\in \{1,\dots,n-1\},
    \quad 
    v_{n}=-\sum_{k=1}^{n-1}v_{k}<0.
  \end{equation*}
  Since we are only interested in the sign, we may assume $v$ to be
  scaled so that 
  \begin{equation*}
    \sum_{k=1}^{n-1}v_{k}=1, \qquad v_{n}=-1.
  \end{equation*}
  Clausen's integral has the asymptotic behavior
  \begin{equation}
    \label{eq:cl_asymp}
    \Cl(x)= -x \log|x|+x+o(x)\quad\text{as}\quad x\rightarrow 0.
  \end{equation}
  This can be seen by considering
  \begin{equation*}\textstyle
    \Cl(x)=-\int_0^x \log\left|(2\sin \frac{t}{2})/t\right|\,dt
    -\int_0^x \log|t|\,dt.
  \end{equation*}
  Using~\eqref{eq:cl_asymp} and the $2\pi$-periodicity of Clausen's
  integral, one obtains
  \begin{equation*}
    \begin{split}
      f_{\ell}(2\pi e_{n}+tv)-f_{\ell}(2\pi e_{n}) 
      &=
      \sum_{k=1}^{n}\big(-tv_{k}\log|v_{k}|+tv_{k}\log\ell_{k}\big) + o(t)\\
      &=
      -\sum_{k=1}^{n-1}tv_{k}\log\frac{v_{k}}{\ell_{k}}-t\log\ell_{n}+o(t),
    \end{split}
  \end{equation*}
  and hence
  \begin{equation*}
    \frac{d}{dt}\Big|_{t=0}f(2\pi e_{n}+tv)=
    -\sum_{k=1}^{n-1}v_{k}\log\frac{v_{k}}{\ell_{k}}-\log\ell_{n}.
  \end{equation*}

  Now we invoke the information inequality~\eqref{eq:information} for
  the discrete probability distributions $(v_1,\dots,v_{n-1})$ and
  $(\ell_1,\dots,\ell_{n-1})/\sum_{k=1}^{n-1} \ell_k$.  Thus,
  \begin{equation*}
    \frac{d}{dt}\Big|_{t=0}f(2\pi e_{n}+tv)
    =
    \underbrace{-\sum_{k=1}^{n-1} v_k \log\left(\frac{v_k}{\ell_k/\sum_{m=1}^{n-1} \ell_m}\right)}_{\leq 0}
    +\log \left(\frac{\sum_{k=1}^{n-1} \ell_k}{\ell_n}\right)
  \end{equation*}

  If $\ell_n \geq \sum_{k=1}^{n-1} \ell_k$, then 
  \begin{equation*}
    \frac{d}{dt}\Big|_{t=0}f(2\pi e_{n}+tv)\leq 0.
  \end{equation*}
  With the concavity of $f_{\ell}$ (Proposition~\ref{prop:concavity}),
  this implies that $f_\ell$ attains its maximum on $\bar D_{n}$ at $(0,
  \ldots, 0, 2\pi)$.
  
  If, on the other hand, $\ell_n < \sum_{k=1}^{n-1} \ell_k$, then we
  obtain, for $v_k=\ell_k/\sum_{m=1}^{n-1} \ell_m$, 
  \begin{equation*}
    \frac{d}{dt}\Big|_{t=0}f(2\pi e_{n}+tv) > 0.
  \end{equation*}
  This implies that $f_\ell$ does not attain its maximum at $(0,
  \ldots, 0, 2\pi)$.
\end{proof}

\noindent%
This completes the proof of Theorem~\ref{thm:euc}.

\begin{remark}
  \label{rem:hyperbolic_volume}
  The function $V_{n}$ has the following interpretation in terms of
  hyperbolic volume~\cite{Milnor1982}. Consider a Euclidean cyclic
  $n$-gon with central angles $\alpha_{1},\ldots,\alpha_{n}$. Imagine
  the Euclidean plane of the polygon to be the ideal boundary of
  hyperbolic $3$-space in the Poincar\'e upper half-space model. Then
  the vertical planes through the edges of the polygon and the
  hemisphere above its circumcircle bound a hyperbolic pyramid with
  vertices at infinity. Its volume is
  $\frac{1}{2}V_{n}(\alpha_{1},\ldots,\alpha_{n})$. Together with Schl\"afli's
  differential volume equation (rather, Milnor's generalization that
  allows for ideal vertices~\cite{milnor_schlfli_1994}), this provides
  another way to prove Proposition~\ref{prop:radius_crit_pt}.
\end{remark}

\section{Spherical polygons. Proof of Theorem~\ref{thm:sph}}
\label{sec:proof_sph}

The polygon inequalities~\eqref{eq:polygon} are clearly necessary for
the existence of a spherical cyclic polygon because every side is a
shortest geodesic. That inequality~\eqref{eq:polygon_sph} is also
necessary was already noted in the introduction. It remains to show
that these inequalities are also sufficient, and that the polygon is
unique.

We reduce the spherical case to the Euclidean one as shown in Figure~\ref{fig:sph_construction}.
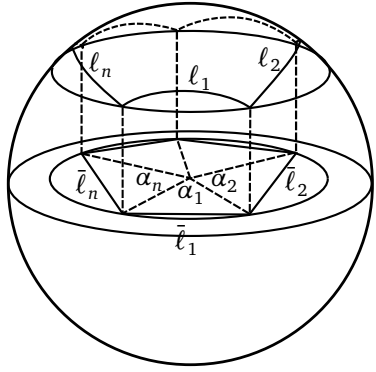
\begin{figure}[h]
\labellist
\small\hair 2pt
 \pinlabel {$\ellbar_1$} [ ] at 72 53
 \pinlabel {$\ellbar_2$} [ ] at 113 75
 \pinlabel {$\ellbar_n$} [ ] at 35 74
 \pinlabel {$\alpha_2$} [ ] at 86 74.5
 \pinlabel {$\alpha_1$} [ ] at 73.5 70
 \pinlabel {$\alpha_n$} [ ] at 58 75
 \pinlabel {$\ell_1$} [ ] at 77 114
 \pinlabel {$\ell_2$} [ ] at 103 121
 \pinlabel {$\ell_n$} [ ] at 40 119
\endlabellist
\centering
\includegraphics[scale=1.0]{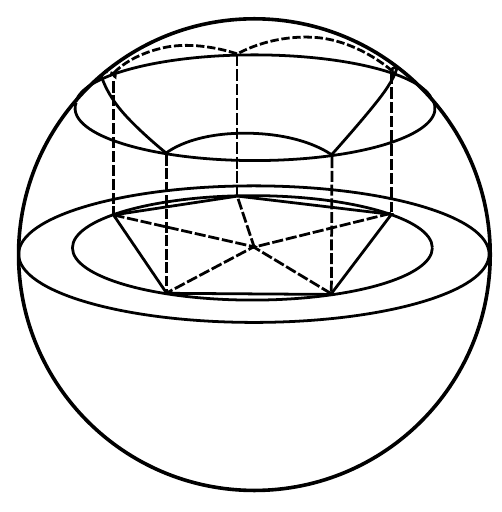}
\caption{Spherical and Euclidean polygons}
\label{fig:sph_construction}
\end{figure}
Connecting the vertices of a spherical cyclic polygon with line segments in the ambient Euclidean space, one obtains a Euclidean cyclic polygon whose circumradius is smaller than $1$. Conversely, every Euclidean polygon inscribed in a circle of radius less than $1$ corresponds to a unique spherical cyclic polygon. The spherical side lengths $\ell$ are related to the Euclidean lengths $\ellbar$ by 
\begin{equation}
\label{eq:lambda_to_ell_sph}
{\ellbar} = 2 \sin\frac{\ell}{2}.  
\end{equation}
It remains to show the following two propositions:

\begin{proposition}
  \label{prop:sph_to_euc_ineq}
  If the spherical lengths $\ell\in\R_{>0}^{n}$ satisfy the inequalities~\eqref{eq:polygon} and~\eqref{eq:polygon_sph}, then the Euclidean lengths $\ellbar$ defined by~\eqref{eq:lambda_to_ell_sph} satisfy the inequalities~\eqref{eq:polygon} as well. By Theorem~\ref{thm:euc} there is then a unique Euclidean cyclic polygon $P_{\ellbar}$ with side lengths $\ellbar$. 
\end{proposition}

\begin{proposition}
  \label{prop:radius}
  The circumradius $\bar R$ of the polygon $P_{\ellbar}$ of
  Proposition~\ref{prop:sph_to_euc_ineq} is strictly less than $1$.
\end{proposition}

We will use the following estimate for a sum of sines in the proof of
Proposition~\ref{prop:sph_to_euc_ineq}:

\begin{lemma}[Sum of Sines]
\label{lem:sin_of_sum}
If $\beta_1, \ldots, \beta_n\in\R_{\geq 0}$ satisfy $\sum_{k=1}^n \beta_k \leq \pi$, then 
\begin{equation}
\label{ineq:sinsum}
\sin\Bigg(\sum_{k=1}^n \beta_k\Bigg) \leq \sum_{k=1}^n \sin\beta_k.
\end{equation}
\end{lemma}

\begin{proof}[Proof of Lemma~\ref{lem:sin_of_sum}]
  By induction on $n$, the base case $n=1$ being trivial. For the inductive step, use the addition theorem,
  \begin{equation*}
    \sin\Bigg(\sum_{k=1}^{n+1} \beta_k\Bigg)=
    \sin\Bigg(\sum_{k=1}^n \beta_k\Bigg)\cos\beta_{n+1}
    +\cos\Bigg(\sum_{k=1}^n \beta_k\Bigg)\sin\beta_{n+1},
  \end{equation*}
  and note that the cosines are $\leq 1$.
\end{proof}

\begin{remark}
  The statement of Lemma~\ref{lem:sin_of_sum} can be strengthened. Equality holds in~\eqref{ineq:sinsum} if and only if at most one $\beta_{k}$ is greater than zero. This is easy to see, but we do not need this stronger statement in the following proof. 
\end{remark}

\begin{proof}[Proof of Proposition~\ref{prop:sph_to_euc_ineq}]
  Suppose $\ell_1, \ldots, \ell_n\in\R_{>0}$ satisfy
  the polygon inequalities~\eqref{eq:polygon}
  and~\eqref{eq:polygon_sph}. We need to show that $\ellbar_1, \ldots,
  \ellbar_n$ defined by~\eqref{eq:lambda_to_ell_sph} satisfy
  \begin{equation}
    \label{eq:ellbar_ineq}
    \ellbar_{k}<\sum_{i\not=k}\ellbar_{i}.
  \end{equation}
  To this end, we will show that 
  \begin{equation}
    \label{eq:sin_ineq}
    \sin\frac{\ell_k}{2}  < \sin\Bigg(\sum_{i\not=k} \frac{\ell_i}{2}\Bigg),
  \end{equation}
  from which inequality~\eqref{eq:ellbar_ineq} follows by Lemma~\ref{lem:sin_of_sum}.
  To prove inequality~\eqref{eq:sin_ineq}, we consider two
  cases separately.

  \smallskip
  \begin{compactitem}
  \item $\sum_{i\not=k}\ell_i \leq \pi$.
  Inequality~\eqref{eq:sin_ineq} simply follows from the polygon
  inequality $\ell_{k}<\sum_{i\not=k}\ell_{i}$ and the monotonicity of
  the sine function on the closed interval $[0,\frac{\pi}{2}]$.
  \item $\sum_{i\not=k}\ell_i \geq \pi$. Note that
  $2\pi>\sum_i\ell_i$ implies
  $2\pi-\ell_{k}>\sum_{i\not=k}\ell_i$, and hence
  \begin{equation}
    2\pi>2\pi-\ell_{k}>\sum_{i\not=k}\ell_{i}\geq\pi.
  \end{equation}
  Inequality~\eqref{eq:sin_ineq} follows from
  $\sin\frac{\ell_k}{2}=\sin(\pi-\frac{\ell_k}{2})$ and the
  monotonicity of the sine function on the closed interval
  $[\frac{\pi}{2},\pi]$.
  \end{compactitem}

  \smallskip\noindent%
  This completes the proof of~\eqref{eq:sin_ineq} and hence the proof of Proposition~\ref{prop:sph_to_euc_ineq}.
  \end{proof}

\begin{proof}[Proof of Proposition~\ref{prop:radius}]
  Let $\alpha_{k}$ be the central angles of the Euclidean cyclic
  polygon $P_{\bar\ell}$. Then
  \begin{equation}\label{eq:radius_sph}
    \sin\frac{\ell_k}{2} = \frac{\ellbar_k}{2} = \bar{R} \sin\frac{\alpha_k}{2},
  \end{equation}
  by~\eqref{eq:eucl_radius} and~\eqref{eq:lambda_to_ell_sph}. Note
  that $\alpha_{k}$ are the central angles of both the Euclidean and
  the spherical polygon (provided it exists). We consider two cases
  separately.

  First, suppose that $\alpha_k \leq \pi$ for all $k$. Since $ \sum_k
  \ell_k < 2 \pi = \sum_k \alpha_k$, there is some $k$ such that
  $\ell_k < \alpha_k$. 
  Then $\sin\frac{\ell_k}{2} < \sin\frac{\alpha_k}{2}$, and 
  equation~\eqref{eq:radius_sph} implies that
  $\bar{R} < 1$.
  
  Otherwise, since $\sum_k\alpha_k = 2\pi$,
  there is exactly one $i$ such that $\alpha_i > \pi$, and $\alpha_k <
  \pi$ for all $k\neq i$. By symmetry, it is enough to consider the
  case 
  \begin{equation*}
    \alpha_1 > \pi,\quad\alpha_k < \pi\quad\text{for}\quad k\in\{2,\ldots,n\}.
  \end{equation*}
  For future reference, we note that $\alpha_1 > \pi$ implies that
  $\bar\ell_{1}$ is the longest side of
  $P_{\bar\ell}$. (Use~\eqref{eq:radius_sph} and the monotonicity of
  the sine function.)  

  We will show $\bar{R}<1$ by induction on $n$.
  First, assume $n=3$. Then~\eqref{eq:sin_ineq} says 
  \begin{equation*}
    \sin\frac{\ell_1}{2}<\sin\frac{\ell_2+\ell_3}{2}\,.
  \end{equation*}
  By~\eqref{eq:radius_sph} and using
  $2\pi-\alpha_{1}=\alpha_{2}+\alpha_{3}$, we have
  \begin{equation}
    \label{eq:sin_lambda/2}
    \sin\frac{\ell_1}{2}
    = \bar{R}\sin\frac{\alpha_2}{2}\cos\frac{\alpha_3}{2} + \bar{R}
    \cos\frac{\alpha_2}{2}\sin\frac{\alpha_3}{2},
  \end{equation}
  and
  \begin{equation}
    \begin{split}
      \sin\frac{\ell_2+\ell_3}{2}
      &=\sin\frac{\ell_2}{2}\cos\frac{\ell_3}{2} + 
      \cos\frac{\ell_2}{2}\sin\frac{\ell_3}{2}\\
      &= \bar{R}\sin\frac{\alpha_2}{2}\cos\frac{\ell_3}{2} +
      \bar{R}\cos\frac{\ell_2}{2}\sin\frac{\alpha_3}{2}\;.
    \end{split}
    \label{eq:sin_(lambda+lambda)/2}
  \end{equation}
  For at least one $k \in \{2,3\}$, $\cos\frac{\alpha_k}{2} <
  \cos\frac{\ell_k}{2}$ and hence $\sin\frac{\alpha_{k}}{2} >
  \sin\frac{\ell_{k  }}{2}$. Equation~\eqref{eq:radius_sph} implies
  $\bar{R}<1$.
  
  Now assume that $\bar{R}<1$ has already been shown if $P_{\bar\ell}$
  has at most $n$ sides. Suppose $P_{\bar\ell}$ has $n+1$ sides. The
  idea of the following argument is to cut off a triangle with sides $\ellbar_{n}$,
  $\ellbar_{n+1}$, and $\bar{\lambda} = 2\bar{R}\sin \frac{\alpha_n +
    \alpha_{n+1}}{2}$.
  Since $\bar{\lambda}\leq\ellbar_{1}$ (the longest side), and
  $\bar\ell_{1}\leq 2$ by~\eqref{eq:lambda_to_ell_sph}, we may define
  $\lambda = 2 \arcsin\frac{ \bar{\lambda}}{2}$.
  Now assume $\bar{R} \geq 1$. Then, 
  by the inductive hypothesis, the polygon
  inequalities~\eqref{eq:polygon} or~\eqref{eq:polygon_sph} are
  violated for the cut-off triangle and the remaining
  $n$-gon. Inequality~\eqref{eq:polygon_sph} cannot be violated
  because it was assumed to hold for
  $\ell_{1},\ldots,\ell_{n+1}$. Hence,
  \begin{equation*}
    \ell_1 \geq \ell_2 + \dots + \ell_{n-1} + \lambda\quad\text{and}\quad\lambda \geq \ell_n + \ell_{n+1}.
  \end{equation*}
  This implies $\ell_1 \geq \ell_2 + \dots + \ell_{n+1}$. Conversely,
  if~\eqref{eq:polygon} and~\eqref{eq:polygon_sph} hold, then $\bar{R} < 1$.
  This completes the proof of Proposition~\ref{prop:radius}.
\end{proof}

\section{Hyperbolic polygons. Proof of Theorem~\ref{thm:hyp}}
\label{sec:proof_hyp}

The polygon inequalities~\eqref{eq:polygon} are clearly necessary for
the existence of a hyperbolic cyclic polygon, because every side is a
shortest geodesic. It remains to show that they are also sufficient,
and that the polygon is unique, i.e.,
Proposition~\ref{prop:ex_uniq_hyp}. First, we review
some basic facts from hyperbolic geometry.

As in the spherical case (Section~\ref{sec:proof_sph}), we will
connect vertices by straight line segments in the ambient vector
space. But instead of the sphere, we consider the hyperbolic plane in
the hyperboloid model,
\begin{equation*}
 \HH^2=\{x\in \R^{2,1} \,|\, \langle x,x\rangle=-1,\, x_3>0\},
\end{equation*}
where $\R^{2,1}$ denotes the vector space $\R^{3}$ equipped with the
scalar product
\begin{equation*}
 \langle x,y\rangle=x_1y_1+x_2y_2-x_3y_3,
\end{equation*}
and lengths and angles in $\HH^{2}$ are measured using the Riemannian
metric induced by this scalar product.

Straight lines (i.e., geodesics) in $\HH^{2}$ are the intersections of
$\HH^{2}$ with $2$-dimensional subspaces of $\R^3$.  The length
$\ell$ of the geodesic segment connecting points $p,q\in\HH^{2}$ is
determined by
\begin{equation*}
 \cosh\ell = -\langle p,q\rangle.
\end{equation*}
The length of the straight line segment connecting points $p,q\in\HH^2$ in the
ambient $\R^{2,1}$ is
\begin{equation*}
\ellbar=\sqrt{\langle p-q,p-q\rangle}.
\end{equation*}
This chordal length $\ellbar$ and the hyperbolic distance $\ell$ are related by
\begin{equation}\label{eq:hypLength}
\frac{\ellbar}{2}=\sinh\frac{\ell}{2}.
\end{equation}

An affine plane in $\R^{2,1}$ is called spacelike, lightlike, or
timelike, if the restriction of the scalar product
$\langle\cdot,\cdot\rangle$ to (the tangent space of) the affine plane
is positive definite, positive semidefinite, or indefinite,
respectively. In terms of the standard Euclidean metric on $\R^{3}$, a
plane is spacelike, lightlike, or timelike if its slope is less than,
equal to, or greater than $45^{\circ}$.

A curve of intersection of $\HH^{2}$ with an affine plane in
$\R^{2,1}$ that does not contain~$0$ is a hyperbolic circle, a
horocycle, or a hypercycle, depending on whether the plane is
spacelike, lightlike, or timelike.

Thus, connecting the vertices of a hyperbolic cyclic polygon by
straight line segments in the ambient $\R^{2,1}$, one obtains a planar
polygon in $\R^{2,1}$, but the intrinsic geometry of the plane will
only be Euclidean if the hyperbolic polygon is inscribed in a circle
(see Figure~\ref{fig:hyp_polygon_circle2}). If the polygon is
inscribed in a horocycle or hypercycle, then the geometry of the plane
will be degenerate with signature $(+,0)$ or a $1+1$-spacetime with
signature $(+,-)$, respectively.
\begin{figure}[p]
\labellist
\small\hair 2pt
 \pinlabel {$\ellbar_n$} [ ] at 127 100
 \pinlabel {$\ellbar_1$} [ ] at 94 100
 \pinlabel {$\ellbar_2$} [ ] at 71 85
 \pinlabel {$\ell_n$} [ ] at 121 64
 \pinlabel {$\ell_1$} [ ] at 100 84
 \pinlabel {$\ell_2$} [ ] at 60 62
\endlabellist
\centering
\includegraphics[scale=1.0]{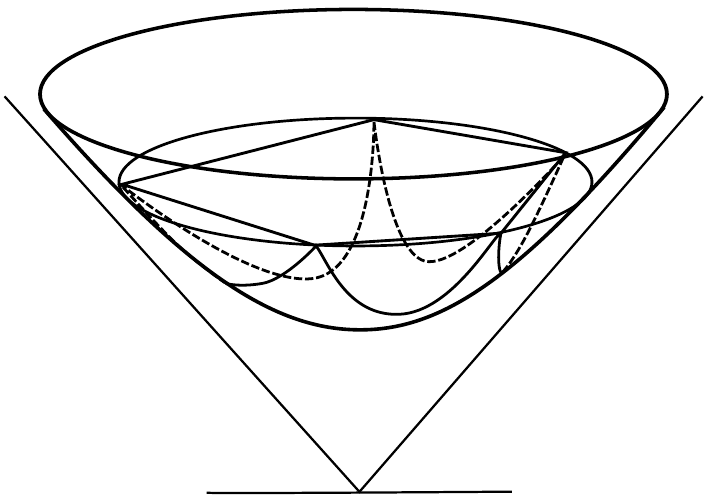}
\caption{Hyperbolic polygon inscribed in a circle, shown in the
  hyperboloid model}
\label{fig:hyp_polygon_circle2}
\end{figure}

\begin{proposition}
\label{prop:ellbar_ineq}
Let $P_{\ell}$ be a hyperbolic cyclic polygon with side lengths
$\ell_{1},\ldots,\ell_n\in\R_{>0}$,
%satisfy the polygon inequalities~\eqref{eq:polygon}. 
and let $\ellbar_{k}$ be the chordal
lengths~\eqref{eq:hypLength}. If $P_{\ell}$ is inscribed in
%$\frac{\ellbar_k}{2}=\sinh\frac{\ell_k}{2}$, the polygon
 \begin{compactenum}[(i)]
  \item a circle then
  \begin{equation}\label{eq:hypCircle}
   \ellbar_{k}<\sum_{\substack{i=1\\i\not=k}}^{n}\ellbar_{i}
   \quad\text{for all }k.
  \end{equation}
  \item a horocycle then
  \begin{equation}\label{eq:hypHoro}
   \ellbar_{k}=\sum_{\substack{i=1\\i\not=k}}^{n}\ellbar_{i}
   \quad\text{for one }k.
  \end{equation}
  \item a hypercycle then
  \begin{equation}\label{eq:hypCurve}
   \ellbar_{k}>\sum_{\substack{i=1\\i\not=k}}^{n}\ellbar_{i}
   \quad\text{for one }k.
  \end{equation}
 \end{compactenum}
\end{proposition}

\begin{proof}
  (i) If $P_{\ell}$ is inscribed in a circle, then the chordal polygon
  obtained by connecting the vertices of $P_{\ell}$ by straight line
  segments in $\R^{2,1}$ is a Euclidean polygon. Hence, its side
  lengths $\bar\ell$ satisfy~\eqref{eq:hypCircle}.

  (ii) If $P_{\ell}$ is inscribed in a horocycle, then the chordal
  length $\bar\ell_{k}$ of a side is equal to the length of the arc of
  the horocycle between its vertices (see
  Figure~\ref{fig:horocycle}). Since one horocycle arc comprises all
  others, this implies~\eqref{eq:hypHoro}.
  
  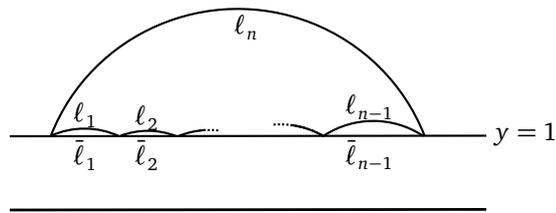
\begin{figure}[p]
    \labellist
    \small\hair 2pt
    \pinlabel {$\ell_1$} [ ] at 32 40
    \pinlabel {$\ell_2$} [ ] at 55 39
    \pinlabel {$\ell_{n-1}$} [ ] at 138 43
    \pinlabel {$\ell_n$} [ ] at 92 73
    \pinlabel {$\ellbar_1$} [ ] at 32 25
    \pinlabel {$\ellbar_2$} [ ] at 55 25
    \pinlabel {$\ellbar_{n-1}$} [ ] at 138 25
    \pinlabel {$y=1$} [ ] at 196 33
    \endlabellist
    \centering
    \includegraphics[scale=1.0]{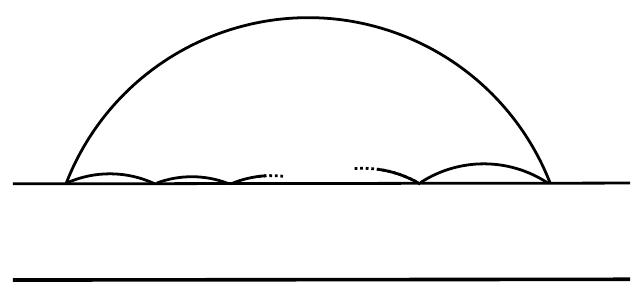}
    \caption{Polygon inscribed in a horocycle, shown in the Poincar\'e
      half-plane model. Here, $\bar\ell_{n}=\sum_{i\not=n}\bar\ell_{i}$ is
      the largest chordal length. Since all horocycles are congruent,
      we may without loss of generality assume that the polygon is
      inscribed in the horocycle $y=1$.}
    \label{fig:horocycle}
  \end{figure}

  (iii) If $P_{\ell}$ is inscribed in a hypercycle at distance $R$
  from a geodesic $g$, then the chordal lengths $\bar\ell_k$, the
  hypercycle ``radius'' $R$, and the distances $a_{k}$ between the
  foot points of the perpendiculars from the vertices to $g$ (see
  Figure~\ref{fig:geodesic}) are related by 
  \begin{equation}
    \label{eq:hypGeodesic}
    \frac{\ellbar_k}{2}=
    {\cosh(R)}\sinh\frac{a_k}{2}. 
  \end{equation}
  \begin{figure}[p]
    \labellist
    \small\hair 2pt
    \pinlabel {$\ell_1$} [ ] at 42 96
    \pinlabel {$\ell_2$} [ ] at 76 105
    \pinlabel {$\ell_{n-1}$} [ ] at 141 99
    \pinlabel {$\ell_n$} [ ] at 76 89
    \pinlabel {$a_1$} [ ] at 40 74
    \pinlabel {$a_2$} [ ] at 76 74                                                                                                                                                      
    \pinlabel {$a_{n-1}$} [ ] at 132 74
    \pinlabel {$\dots$} [ ] at 106 88
    \pinlabel {$R$} [ ] at 100 98
    \pinlabel {$R$} [ ] at 63 98
    \pinlabel {$g$} [ ] at 10 73
    \endlabellist
    \centering
    \includegraphics[scale=1.0]{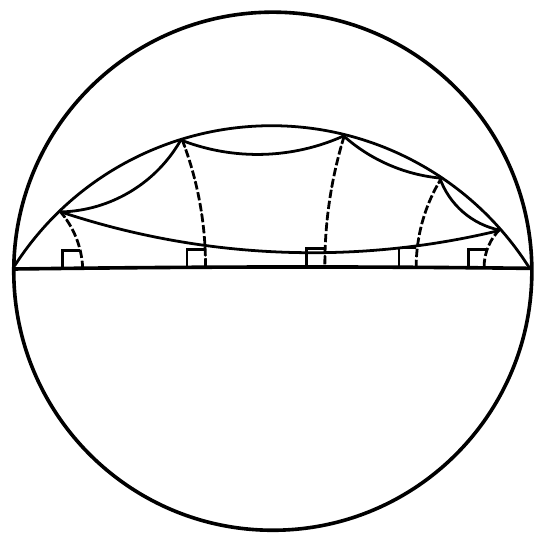}
    \caption{Polygon inscribed in a hypercycle, shown in the
      Poincar\'e disk model. Here, $\ell_{n}$ and $\bar\ell_{n}$ are
      the largest side length and chordal length, respectively.}
    \label{fig:geodesic}
  \end{figure}
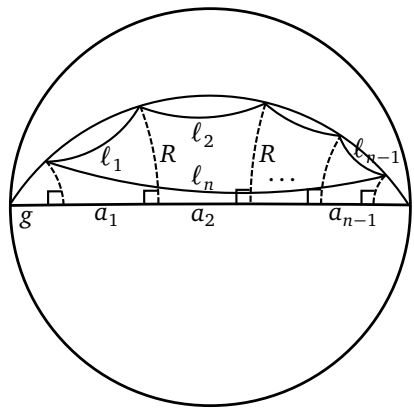
  Since one of the segments of $g$ comprises all others,
  \begin{equation*}
    a_{k}=\sum_{i\not=k}a_{i}\quad\text{for one } k.
  \end{equation*}
  With~\eqref{eq:hypGeodesic} this implies~\eqref{eq:hypCurve}:
  \begin{equation*}
    \frac{\ellbar_{k}}{2} 
    =
    \cosh(R)\sinh\Big(\sum_{i\not=k}\frac{a_i}{2}\Big)
    >
    \sum_{i\not=k}\cosh(R)\sinh\frac{a_i}{2}=\sum_{i\not=k} \frac{\ellbar_i}{2}\,,
  \end{equation*}
  where we have used the inequality $\sinh(x+y)>\sinh(x)+\sinh(y)$, which
  holds for positive $x$, $y$. This follows immediately from the
  addition theorem for the hyperbolic sine function.

  This completes the proof of Proposition~\ref{prop:ellbar_ineq}.
\end{proof}
 
\begin{proposition}
  \label{prop:ex_uniq_hyp}
  If $\ell \in \R^n_{>0}$ satisfies the polygon
  inequalities~\eqref{eq:polygon}, then there exists a unique
  hyperbolic cyclic polygon with these side lengths.
\end{proposition}

\begin{proof} 
  Suppose $\ell \in \R^n_{>0}$ satisfies the polygon
  inequalities~\eqref{eq:polygon}. Let $\bar\ell$ be the corresponding
  chordal lengths~\eqref{eq:hypLength}. We will treat each case of
  Proposition~\ref{prop:ellbar_ineq} separately. In each case, we will
  tacitly use Proposition~\ref{prop:ellbar_ineq} and its proof. Our
  treatment of case (iii) is analogous to Penner's
  proof~\cite{Penner1987} of Theorem~\ref{thm:euc} (his Theorem~6.2).

  (i) If the chordal lengths $\bar\ell$ satisfy
  condition~\eqref{eq:hypCircle}, then the existence and uniqueness of
  a hyperbolic cyclic polygon with side lengths~$\ell$ follows from
  the existence and uniqueness of a Euclidean cyclic polygon with side
  lengths~$\bar\ell$, i.e., from Theorem~\ref{thm:euc} (see
  Figure~\ref{fig:hyp_polygon_circle2}).

  (ii) If the chordal lengths $\bar\ell$ satisfy
  condition~\eqref{eq:hypHoro}, then the corresponding hyperbolic
  cyclic polygon can be constructed by marking off the lengths
  $\bar\ell_{i}$ for $i\not=k$ along a horocycle (see
  Figure~\ref{fig:horocycle}). To see the uniqueness claim, note
  that all horocycles are congruent.

  (iii) It remains to consider the case that the chordal lengths
  $\bar\ell$ satisfy condition~\eqref{eq:hypCurve}. For simplicity, we
  will assume that $\ell_{n}$ is the largest side length. Then
  $\bar\ell_{n}$ is the largest chordal length and
  condition~\eqref{eq:hypCurve} says
  \begin{equation}
    \bar\ell_{n}>\sum_{k=1}^{n-1}\bar\ell_{k}\,.
  \end{equation}

  Now suppose $P_{\ell}$ is a hyperbolic polygon with side lengths $\ell$
  that is inscribed in a hypercycle at distance $R$ from its
  geodesic $g$, and let 
  \begin{equation}
    \label{eq:barR}
    \bar R=\cosh R. 
  \end{equation}
  Then the distances $a_{k}$ between the foot points (see
  Figure~\ref{fig:geodesic}) satisfy
  \begin{equation*}
    a_{n}=\sum_{k=1}^{n-1}a_{k}, 
  \end{equation*}
  Using~\eqref{eq:hypGeodesic}, one obtains
  \begin{equation}
    \label{eq:arsinh_sum}
    \arsinh\Big(\frac{\bar\ell_{n}}{2\bar{R}}\Big)=
    \sum_{k=1}^{n-1}\arsinh\Big(\frac{\bar\ell_{k}}{2\bar{R}}\Big).
  \end{equation}
  Conversely, if, for given $\ell$, a number $\bar{R}>1$
  satisfies~\eqref{eq:arsinh_sum} then $R$ defined by~\eqref{eq:barR}
  is the correct hypercycle distance. More precisely, one can then
  construct a hyperbolic cyclic polygon with side lengths $\ell$ by
  marking off the distances $a_{1},\ldots,a_{n-1}$ determined
  by~\eqref{eq:hypGeodesic} along a geodesic and intersect the
  perpendiculars in the marked points with a hypercycle at distance
  $R$ (see Figure~\ref{fig:geodesic}).

  It remains to show that there is exactly one $\bar{R}>1$
  satisfying~\eqref{eq:arsinh_sum}. To this end, consider the function 
  \begin{equation}
    \label{eq:Phi}
    \Phi(x)=\arsinh\Big(\frac{\bar\ell_{n}}{2x}\Big)
    -\sum_{k=1}^{n-1}\arsinh\Big(\frac{\bar\ell_{k}}{2x}\Big).
  \end{equation}
  We need to show that $\Phi$ has exactly one zero in the interval
  $(1,\infty)$. Using~\eqref{eq:hypLength}, we see
  \begin{equation}
    \label{eq:Phi1}
    \Phi(1) = \frac{1}{2}\Big(\ell_{n}-\sum_{k=1}^{n-1}\ell_{k}\Big)<0.
  \end{equation}
  For $x\rightarrow\infty$,
  \begin{equation}
    \label{eq:Phi_infty_asymptotic}
    \Phi(x)=\frac{1}{2x}\Big(\bar\ell_{n}-\sum_{k=1}^{n-1}\bar\ell_{k}\Big)
    + o\Big(\frac{1}{x}\Big),
  \end{equation}
  so 
  \begin{equation*}
    \Phi(x)>0\quad\text{for large }x.
  \end{equation*}
  By continuity, $\Phi$ has at least one zero in the interval
  $(1,\infty)$.  

  Finally, we will show that the derivative
  \begin{equation}
    \label{eq:Phiprime}
    \Phi'(x)=
    \frac{1}{x}\Bigg(
    -\frac{\frac{\bar\ell_{n}}{2x}}{\sqrt{1+\big(\frac{\bar\ell_{n}}{2x}\big)^{2}}}
    +
    \sum_{k=1}^{n-1}
    \frac{\frac{\bar\ell_{k}}{2x}}{\sqrt{1+\big(\frac{\bar\ell_{k}}{2x}\big)^{2}}}
    \Bigg),
  \end{equation}
  is positive at the positive zeroes of $\Phi$. This implies that
  $\Phi$ has at most one zero in $\R_{>0}$. Let us define
  \begin{equation*}
    a_{k}(x)=\arsinh\Big(\frac{\bar\ell_{k}}{2x}\Big),
  \end{equation*}
  so
  \begin{equation*}
    \Phi(x)=a_{n}(x)-\sum_{k=1}^{n-1}a_{k}(x),
  \end{equation*}
  and
  \begin{equation*}
    \Phi'(x)=\frac{1}{x}\Big(
    -\tanh a_{n}(x)+\sum_{k=1}^{n-1}\tanh a_{k}(x)
    \Big).
  \end{equation*}
  The claim follows from the inequality 
  \begin{equation*}
    \tanh\sum_{j=1}^{m}a_{j}<\sum_{j=1}^{m}\tanh a_{j},
  \end{equation*}
  which holds for $m\geq 2$ and positive numbers $a_{j}$. (Use
  induction on $m$ and the addition theorem
  for $\tanh$ for the base case $m=2$.)
  
  This concludes the proof of Proposition~\ref{prop:ex_uniq_hyp}.
\end{proof}

\section{Concluding remarks on $\boldsymbol{1+1}$
  spacetime}
\label{sec:concluding}

The scalar product of $\R^{1,1}$ is 
\begin{equation*}
  \langle x,y\rangle_{1,1}=x_{1}y_{1}-x_{2}y_{2},
\end{equation*}
and the length of a spacelike vector $x$ is $\ell=\sqrt{\langle
  x,x\rangle_{1,1}}$. The proof of Theorem~\ref{thm:hyp} for polygons
inscribed in hypercycles (Section~\ref{sec:proof_hyp}) also proves
the following theorem about ``cyclic'' polygons in $1+1$ spacetime.

\begin{theorem}
  \label{thm:11}
  There exists a polygon in $\R^{1,1}$ with $n\geq 3$ spacelike sides
  with lengths $\ell_{1},\ldots,\ell_n>0$ that is inscribed in one
  branch of a hyperbola $\langle x,x\rangle_{1,1}=-R^{2}$ if and only
  if
  \begin{equation}
    \label{eq:11ineq}
    \ell_{k}>\sum_{\substack{i=1\\i\not=k}}^{n}\ell_{i}\quad\text{for
      one }k,
  \end{equation}
  and this polygon is unique.
\end{theorem}

Without loss of generality, we will assume that the $n$th side is the
longest, i.e., $k=n$ in~\eqref{eq:11ineq}. Like in the Euclidean case
(Section~\ref{sec:proof_euc}), the construction of such an inscribed
polygon in $\R^{1,1}$ is equivalent to the following analytic problem:
Find a point $a\in\R_{>0}^{n}$ satisfying
\begin{equation}
  \label{eq:11constraint}
  a_{n}=\sum_{i=1}^{n-1}a_{i}
\end{equation}
and
\begin{equation}
  \label{eq:11radius}
  \frac{\ell_{k}}{2}=R\sinh\frac{a_{k}}{2}
\end{equation}
for some $R\in\R$ and all $k\in\{1,\ldots, n\}$. 

This problem admits the following variational formulation. Define the
function $\varphi_{\ell}:\R^{n}\rightarrow \R$ by
\begin{equation*}
  \varphi_{\ell}(a)=\sum_{k=1}^{n-1}\big( 
  \Clh(a_k) +  \log(\ell_k)\,a_k\big)
  -\big(\Clh(a_n) +  \log(\ell_n)\,a_n\big),
\end{equation*}
where $\Clh$ denotes the ``hyperbolic version'' of Clausen's integral,
\begin{equation*}
  \Clh(x)= -\int_0^x \log\Big|2\sinh \frac{t}{2}\Big|\,dt.
\end{equation*}
(This notation is not standard.) The function $\Clh(x)$ can be
expressed in terms of the real part of the dilogarithm function: 
\begin{equation*}
  \Clh(x)=\operatorname{Re}\operatorname{Li}_{2}(e^{x})+\frac{x^{2}}{4}-\frac{\pi^{2}}{6}.
\end{equation*}
Like in the Euclidean case, one sees that $a\in\R_{>0}^{n}$ is a
critical point of $\varphi_{\ell}$ under the
constraint~\eqref{eq:11constraint} if and only if there is an $R$
satisfying equations~\eqref{eq:11radius}. However, the function
$\varphi_{\ell}$ is neither concave nor convex on the
subspace~\eqref{eq:11constraint}, so any proof of Theorem~\ref{thm:11}
(or the hypercycle case of Theorem~\ref{thm:hyp}) based on this
variational principle would have to be more involved.

\paragraph{Acknowledgement.}
This research was supported by DFG SFB/TR 109 ``Discretization in Geometry
and Dynamics''.

\bibliographystyle{abbrv}
\bibliography{cyclic}

\end{document}